\documentclass[]{interact}
\usepackage{epstopdf}
\usepackage[caption=false]{subfig}

\usepackage[numbers,sort&compress]{natbib}
\bibpunct[, ]{[}{]}{,}{n}{,}{,}
\renewcommand\bibfont{\fontsize{10}{12}\selectfont}
\makeatletter
\def\NAT@def@citea{\def\@citea{\NAT@separator}}
\makeatother

\theoremstyle{plain}
\newtheorem{theorem}{Theorem}[section]

\newtheorem{corollary}[theorem]{Corollary}
\newtheorem{proposition}[theorem]{Proposition}

\theoremstyle{definition}
\newtheorem{definition}[theorem]{Definition}
\newtheorem{example}[theorem]{Example}

\theoremstyle{remark}

\def\diag{\mathop{\rm diag}}
\def\Diag{\mathop{\rm Diag}}
\def\inertia{\mathop{\rm Inertia}}
\def\rank{\mathop{\rm rank}}
\def\span{\mathop{\rm span}}
\def\inertia{\mathop{\rm In}}

\newcommand{\rr}{\mathbb{R}}

\def\1{\mathbf{1}}

\newcommand{\lan}{\langle}
\newcommand{\ran}{\rangle}

\def\range{\mathop{\rm col}}
\def\EDM{\mathop{\rm EDM}}
\def\GDM{\mathop{\rm GEDM}}
\def\nul{\mathop{\rm null}}

\def\a{\alpha}
\def\b{\beta}

\def\min{{\rm min}}

\def\tr{{\rm trace}}

\def\L{\wideitilde{L}}

\def\S{\widetilde{S}}

\def\L{\widetilde{L}}

\begin{document}


\title{Generalized Euclidean distance matrices}

\author{
\name{R. Balaji \textsuperscript{a}
, R.B. Bapat\textsuperscript{b} and Shivani Goel\textsuperscript{c}\thanks{CONTACT: Shivani Goel. Email: shivani.goel.maths@gmail.com}}
\affil{\textsuperscript{a}Department of Mathematics, IIT Madras, Chennai, India; \textsuperscript{b}Theoretical Statistics and Mathematics Unit, Indian Statistical Institute, Delhi, India;\textsuperscript{c}Department of Mathematics, IISc Bangalore, Bangalore,India }
}

\maketitle

\begin{abstract}
Euclidean distance matrices ($\EDM$) are symmetric nonnegative matrices with several interesting properties. 
In this article, we introduce a wider class of matrices called generalized Euclidean distance matrices ($\GDM$s) that include $\EDM$s. Each $\GDM$ is an entry-wise nonnegative matrix. A $\GDM$ is not symmetric unless it is an $\EDM$. By some new techniques, we show that
many significant results on Euclidean distance matrices can be extended to generalized Euclidean distance matrices.
These contain results about eigenvalues, inverse, determinant, spectral radius, Moore-Penrose inverse and some majorization inequalities. 
We finally give an application by constructing infinitely divisible matrices using generalized Euclidean distance matrices. 
\end{abstract}


\begin{amscode}
15A57
\end{amscode}

\section{Introduction}
An $n \times n$ real matrix $D=[d_{ij}]$  is a {\em Euclidean distance matrix} (EDM) if there exist vectors
$x^1,x^2,\dotsc,x^n$ in a Euclidean space $(V,\lan \cdot,\cdot \ran)$ such that 
\begin{equation} \label{defn}
d_{ij}= \lan x^i - x^j, x^i-x^j \ran~~~\mbox{for all}~ i,j=1,2,\dotsc,n ; 
\end{equation}
or equivalently,
\[d_{ij}= \|x^i-x^j\|^{2}~~\mbox{for all}~ i,j=1,2,\dotsc,n. \]
Euclidean distance matrices appear in several fields. For instance, in approximation theory, Micchelli \cite{mic} proved  the striking
result: If $x^1,\dotsc,x^n$ are $n$ distinct points in the plane, then 
\[(-1)^{n-1} \det [\sqrt{1+ \|x^i-x^j\|^2}]  >0.\]
As a consequence, there exists a unique function $$f(x)=\sum_{k=1}^{n} c_{k} \sqrt{1+ \|x-x^k\|^{2}}$$
interpolating a given data $y_{1},\dotsc,y_n$ at $x^1,\dotsc,x^n$.
Euclidean distance matrices have several interesting properties. For example, we have the following basic result. 
\begin{theorem}[Menger, Schoenberg.  \cite{Fiedler}] \label{fie}
Let $D=[d_{ij}]$ be an $n \times n$ symmetric matrix with zero diagonal. Then the following are equivalent.
\begin{enumerate}
\item[{\rm 1.}] $D$ is an $\EDM$.
\item[{\rm 2.}] If $(x_1,\dotsc,x_n)'$ is such that $\sum_{i=1}^n x_{i}=0$, then $\sum_{i,j}d_{ij}x_{i}x_{j} \leq 0$. 
\item[{\rm 3.}] Let $\1=(1,1,\dotsc,1)' \in \rr^n$. Then the bordered matrix
$\left[
\begin{matrix}
D & \1 \\
\1' & 0
\end{matrix}
\right]$ has exactly one (simple) positive eigenvalue.
\end{enumerate}  
\end{theorem}

Symmetric matrices with exactly one simple positive eigenvalue are called elliptic matrices. Elliptic matrices play an important role in Alexandrov inequalities for mixed volumes \cite[chapter 5]{Bapat}.
Elliptic matrices are useful in obtaining infinitely divisible matrices. If $[a_{ij}]$ is an elliptic matrix with all entries positive, then $[\displaystyle \frac{1}{a_{ij}^{ r}}]$ is positive semidefinite for all $r>0$, i.e.  $[\frac{1}{a_{ij}}]$ is a infinitely divisible matrix. The main theme of \cite{Bhatia} is to obtain infinitely divisible matrices via elliptic matrices.

 \subsection{Objective of the paper}
In this paper,  we investigate the so-called generalized Euclidean distance matrices ($\GDM$).
All Euclidean distance matrices ($\EDM$) are generalized Euclidean distance matrices.
 If a generalized Euclidean distance matrix is not an $\EDM$, then it is not symmetric.
Despite this fact, we extend many properties of $\EDM$s to  $\GDM$s.
For example, we show that all eigenvalues of a $\GDM$ are real and has at most one positive eigenvalue, null space of a $\GDM$ 
is a subspace of $\1^\perp$, the Moore-Penrose inverse of a $\GDM$  is negative semidefinite on $\1^\perp$ and so on. 
These results are obtained by using new techniques that circumvent the standard arguments on symmetric matrices.

\subsection{Definition of generalized Euclidean distance matrices}
We begin with the following observation.  Let $D=[d_{ij}]$ be an $n \times n$ Euclidean distance matrix. 
In view of (\ref{defn}), if $D=[d_{ij}]$ is a EDM, then 
\[d_{ij}=\|x^i\|^{2} + \|x^j\|^{2} -2 \lan x^i, x^j \ran ~~~\forall i,j=1,2,\dotsc,n. \]
Recall that the Gram matrix of an ordered system of vectors $\{v^1,\dotsc,v^n\}$ in a Euclidean space is the matrix
$[\lan v^i,v^j \ran]$.
As every positive semidefinite matrix is a Gram matrix of some system of vectors in some Euclidean space, it follows that
$D=[d_{ij}]$ is an EDM if and only if there exists an $n \times n$ positive semidefinite matrix $F=[f_{ij}]$ such that
\begin{equation} \label{df11}
d_{ij}=f_{ii}+f_{jj}-2f_{ij}.
\end{equation}
Let $\1$ be the vector of all ones in $\rr^n$. Define $J:=\1 \1'$.  
Then (\ref{df11}) can be rewritten as
\begin{equation} \label{df}
D=\diag(F)J + J \diag(F) -2 F.
\end{equation}
Define 
\begin{equation} \label{df1}
P:=I-\frac{1}{n} J~~\mbox{and}~~ Y=[y_{ij}]:=-\frac{1}{2}PDP,
\end{equation}
where $I$ is the $n \times n$ identity matrix.
As the diagonal entries of $D$ are zero, 
from equation $(\ref{df11})$, it can be verified that 
\[d_{ij}=y_{ii} + y_{jj}-2y_{ij}, \]
or equivalently,
\[D= \diag(Y)J+J \diag(Y)-2Y .\] 
By an easy verification, it follows that $Y$ is positive semidefinite and $Y \1=0$. Thus, a symmetric matrix $D=[d_{ij}]$ is an EDM if and only if 
there exists a positive semidefinite matrix $G=[g_{ij}]$ such that $G \1=0$ (equivalently, row sums and column sums of $G$ are zero) and 
\[d_{ij}=g_{ii} + g_{jj} -2g_{ij}. \] 
Several important properties of Euclidean distance matrices depend on this characterization. 

\begin{definition} \rm \label{lap}
We say that an $n \times n$ real symmetric matrix $L$ is a {\em generalized Laplacian matrix} if $L$ is positive semidefinite
   and $L \1=0$. 
   \end{definition}
   It is easy to note that $L$ is a generalized Laplacian if and only if there exists a positive semidefinite matrix $F$ such that
$L=PFP$, where 
$P$ is the matrix in $(\ref{df1})$.
Laplacian matrices of connected graphs are examples of generalized Laplacian matrices.

\begin{definition}  \rm \label{p}
Let $a$ and $b$ be any two positive numbers and $L$ be an $n \times n$  generalized Laplacian matrix. 
Define
\[d_{ij}=a^2 l_{ii} + b^{2} l_{jj} -2 ab l_{ij}~~~ i,j=1,2,\dotsc,n.\]
We now say that $[d_{ij}]$ is a generalized Euclidean distance matrix ($\GDM$). 
\end{definition}
An easy computation shows that 
\[d_{ji}=b^2 l_{ii} + a^2 l_{jj} -2ab l_{ij}. \]
Thus, $D$ is symmetric if and only if $a=b$ and in this case,
all the diagonal entries of $D$ are zero; hence $D$ will be a Euclidean distance matrix. 
To illustrate the definition, consider the following example.
\begin{example} \label{gedmexample}\rm
Let 
$$L:= [l_{ij}]=
\left[
\begin{array}{cccc }
~~3 & -1 & -2 \\
-1 & ~~3 & -2 \\
-2 & -2 & ~~4 \\
\end{array}
\right].$$ 
Then $L$ is a generalized Laplacian matrix.  Set $a=1$ and $b=3$. For $i,j=1,2,3$,  define 
\begin{equation*}
\begin{aligned}
d_{ij}:&=a^2l_{ii}+b^2 l_{jj} -2abl_{ij}\\
&= l_{ii}+9 l_{jj} -6l_{ij}.
\end{aligned}
\end{equation*}
   Now,
\[D:=[d_{ij}] = \left[
\begin{array}{cccc}
12 & 36 & 51 \\
36 & 12 & 51 \\
43 & 43 & 16
\end{array}
\right]\]
 is a $\GDM$.
\end{example}

We note the following proposition.
\begin{proposition} \label{prop}
 An $n \times n$ matrix $D$ is a $\GDM$ if and only if there exist $x^1,\dotsc,x^n$ in some Euclidean space $(V,\lan \cdot, \cdot \ran)$ such that 
\[d_{ij}=\langle ax^i-bx^j,ax^i-bx^j \rangle ~~\mbox{and}~~\sum_{j=1}^nx^{j}=0.\]
\end{proposition}
\begin{proof}
Let $D=[d_{ij}]$ be a $\GDM$. Then, 
\[d_{ij}=a^2 l_{ii} + b^2 l_{jj} -2ab l_{ij},  \]
where $L=[l_{ij}]$ is a generalized Laplacian matrix and $a,b$ are some positive numbers. Because $L$ is positive semidefinite and $L \1=0$, there exist $x^1,\dotsc,x^n$  in a Euclidean space $(V,\lan \cdot, \cdot \ran)$ such that $$l_{ij}=\langle x^i,x^j \rangle~~\mbox{and}~~\sum_{j=1}^{n}x^j=0. $$
Thus,
\begin{equation*} 
\begin{aligned}
d_{ij}= a^2 \langle x^i,x^i \rangle + b^2 \langle x^j, x^j \rangle -2ab \langle x^i, x^j \rangle =\langle ax^i-bx^j,ax^i-bx^j \rangle.
\end{aligned}
\end{equation*}
The converse is immediate.
\end{proof}

\subsection{Preliminaries}
We fix the notation and mention a few results/definitions that are needed in the sequel.
\renewcommand{\labelenumi}{(N\arabic{enumi})}
\begin{enumerate}
\item \label{(N1)} All matrices considered are real. The transpose of a matrix $A$ is written $A'$. The notation $\1$ will denote the vector of all ones in $\rr^n$, $J$ will be the matrix $\1 \1'$ 
and $\1^{\perp}$ will be the subspace containing all vectors that are orthogonal to $\1$. As usual $e_1,\dotsc,e_n$ will denote the standard orthonormal basis vectors in $\rr^n$; hence the first column of an $n \times n$ matrix $A$ will be $Ae_1$ and so on. We use $P$ to denote the orthogonal projection $I-\frac{1}{n}J$ onto $\1^\perp$.

\item \label{(N2)} The null space of $A$ is denoted by $\nul(A)$ and the column space (range) of $A$ by $\range(A)$.
As usual, $\rho(A)$ will be the spectral radius of $A$ and the Moore-Penorse inverse of $A$ will be written 
$A^\dag$. 

\item \label{(N3)} Let $x:=(x_{1},\dotsc,x_n)'$ and $y:=(y_1,\dotsc,y_n)'$ be any two vectors. Let
$\sigma$ and $\pi$ be permutations on $\{1,\dotsc,n\}$ such that 
\[x_{\sigma(1)} \geq \cdots  \geq x_{\sigma_n}~~\mbox{and}~~y_{\pi_1}\geq \cdots \geq y_{\pi_n}. \]
We say that $x$ is majorized by $y$ if 
$$\sum_{i=1}^{n} x_i= \sum_{i=1}^n y_i~~~\mbox{and}~~\sum_{j=1}^k x_{\sigma(j)} \leq \sum_{j=1}^k y_{\pi(j)}~~~(k=1,\dotsc,n-1).$$
To say that $x$ is majorized by $y$, we use the notation $x \prec y$.

\item \label{(N4)} Given an $n \times n$ matrix $A=[a_{ij}]$, we use $\diag(A)$ to denote the
vector $(a_{11},\dotsc,a_{nn})'$ in $\rr^n$. If $(p_1,\dotsc,p_n)'$ is a vector in $\rr^n$, we use
$\Diag(p_1,\dotsc,p_n)$ to denote the diagonal matrix with diagonal entries $p_1,\dotsc,p_n$.

\item \label{(N5)} Let $H$ be an $n \times n$  
symmetric matrix and the eigenvalues are $\lambda_1\dotsc,\lambda_n$ such that
$$\lambda_1 \geq \cdots \geq \lambda_n.$$ We now define $\lambda(H)$ by $$\lambda(H)=(\lambda_1,\dotsc,\lambda_n)'.$$

\item \label{(N6)} By a theorem of Schur, if $H$ is an $n \times n$ symmetric matrix, then
\begin{equation*} 
\diag(H) \prec \lambda(H).
\end{equation*} If $A$ and $B$ are $n \times n$ symmetric matrices, then a result of Ky Fan states that
$$\lambda(A+B) \prec \lambda(A)+\lambda(B).$$ (see e.g., \cite[Theorem 7.14 and 7.15]{Zhang})

\item \label{(N61)} The inertia of a symmetric matrix $A$ is denoted by $\inertia(A)=(\nu,\delta,\mu)$ where $\nu$ is the number of 
negative eigenvalues of $A$, $\delta$ is the nullity of $A$ and $\mu$ is the number of positive eigenvalues of $A$. 
If $A$ is $n \times n$, $p \in \rr^n$, and if 
\[\widetilde{A}=\left[
\begin{array}{ccc}
A & p \\
p' & 0
\end{array}
\right],
\]
then the generalized Schur complement $\widetilde{A}/A$ is $-p'A^{\dag}p$. 
If $p \in \range(A)$, then 
\[\inertia(\widetilde{A}/A) + \inertia(A)=\inertia(\widetilde{A}). \]
(see \cite[Theorem 2]{Haynsworth})

\item \label{(N7)} An $n \times n$ matrix $A$ is called an ${\bf M}$-matrix, if $A= 
\rho I-S$,
where $S$ is a nonnegative matrix and $\rho \geq \rho(S)$.

\item \label{(N8)} Let $A$ be an $n \times n$ symmetric matrix and $p \in \rr^n$. If
the bordered matrix
\[
\left[
\begin{array}{ccccccc}
A & p \\
p' & 0
\end{array}
\right]
\] has exactly one positive eigenvalue, then 
\[y \in p^\perp \implies y'Ay \leq 0. \] (see \cite[Theorem 2.9]{Ferland})
\end{enumerate}

\section{Results}
 In the sequel, we assume that $D$ is a non-zero $\GDM$ and 
$L$ is the generalized Laplacian such that
\begin{equation} \label{D}
D=a^2 \L J +b^2 J \L -2ab L, 
\end{equation}
where $a$ and $b$ are fixed positive numbers and $\L:=\Diag(L)$.
\subsection{Eigenvalues of a $\GDM$}
We shall now show that all eigenvalues of $D$ are real and in fact $D$ is similar to a symmetric matrix. 
\begin{theorem} \label{ww}
 $D$ is similar to a symmetric matrix and has exactly one positive eigenvalue.
\end{theorem}
\begin{proof}
 If all the diagonal entries of $D$ are zero, then $D$ is a Euclidean distance matrix and in this case, the result is known.
Now assume that $D$ has at least one positive diagonal entry. 
Let the eigenvalues of $L$ be arranged
$$\a_1 \leq \a_2 \leq \cdots \leq \a_n.$$
Since $L$ is positive semidefinite and $L \1=0$, $\a_1=0$. 
Let $U$ be an orthogonal matrix with first column equal to the unit vector $\frac{1}{\sqrt{n}} \1$ and such that 
\[U'LU=\Diag(0,\a_2,\dotsc,\a_n). \]  
Define
\[(s_1,s_2,\dotsc,s_n)' := U'\L JU e_1.\] By an easy verification, \[s_1=\tr(L).\]
Furthermore,  
\[U' \L JU e_i=0 ~~~~i=2,\dotsc,n.\]
Thus, 
\begin{equation} \label{pdp}
\begin{aligned}
U'DU&=a^2 U' \L J U 
+b^2 U' J \L U -2ab U'LU \\
&=\left[
\begin{array}{ccccc}
(a^{2} + b^2)s_{1} & b^2 s_{2} & \hdots & b^2 s_{n} \\
a^{2} s_{2} & -2ab \a_{2} & \hdots & 0\\
\vdots & \vdots & \ddots & \vdots \\
a^2 s_{n} & 0 & \hdots & -2ab \a_{n}
\end{array}
\right].
\end{aligned}
\end{equation}
Define $$W:=\mbox{Diag}(\frac{\displaystyle b}{\displaystyle a},1,\dotsc,1).$$
Now, 
\begin{equation} \label{m4}
W^{-1}U'DUW=
\begin{array}{ccccc}
\left[
\begin{array}{ccccc}
(a^{2} + b^2)s_{1} & ab s_{2} & \hdots & ab s_{n} \\
abs_{2}  & -2ab \a_{2} & \hdots & 0\\
\vdots & \vdots & \ddots & \vdots \\
ab s_{n} & 0 & \hdots & -2ab \a_{n}
\end{array}
\right];
\end{array}
\end{equation}
hence $D$ is similar to a symmetric matrix. 
By interlacing theorem, $W^{-1}U'DUW$ has $n-1$ non-positive eigenvalues. Since  $\tr(D)>0$, we see that 
\[(W^{-1}U'DUW)_{11}>0.\] 
Hence, $W^{-1}U'DUW$ has at least one positive eigenvalue. So, $D$ has exactly one positive eigenvalue. The proof is now complete.
\end{proof}
In the rest of the paper,  we shall use $W$ and $U$ to denote the matrices defined in Theorem $\ref{ww}$. 
We have the following corollary now.
\begin{corollary} \label{bordered}
All the eigenvalues of the bordered matrix \[ \widetilde{D}:=
\left[
\begin{array}{cccc}
D & \1 \\
\1' & 0
\end{array}
\right]
\] are real and $\widetilde{D}$ has exactly one positive eigenvalue.
\end{corollary}
\begin{proof}
Setting
\[K:= \left[
\begin{array}{cccc}
U & 0 \\
0 & 1 \\
\end{array}
\right] \left[
\begin{array}{cccc}
W & 0 \\
0 & 1 \\
\end{array}
\right], 
\]
we find that 
the bordered matrix \[\widetilde{D}:=
\left[
\begin{array}{cccccc}
D & \1 \\
\1' & 0
\end{array}
\right]
\] is similar to 
\begin{equation} \label{Q}
Q:=K^{-1} \widetilde{D} K=
\left[
\begin{array}{cccccc}
W^{-1}U'DUW & W^{-1}U'\1  \\
\1'UW & 0
\end{array}
\right].
\end{equation}
Since $U$ is an orthogonal matrix with first column equal to $\frac{1}{\displaystyle \sqrt{n}}\1$,  
\begin{equation} \label{m5} 
\begin{aligned}
 W^{-1}U'\1 &=\Diag(\frac{a}{b},1,\dotsc,1)(\sqrt{n}e_{1}) \\
  &=(\sqrt{n}\frac{a}{b} ,0,\dotsc,0)' \\
  &=\sqrt{n} \frac{a}{b} e_1.
 \end{aligned}
 \end{equation}
 Put \[\delta:=\sqrt{n}\dfrac{a}{b}.\]
By $(\ref{m4})$ and $(\ref{m5})$,
\begin{equation} \label{Qe}
\begin{aligned}
Q&=\left[
\begin{array}{ccccc}
W^{-1}U'DUW & \delta e_{1} \\
\frac{1}{\delta} e_{1}' & 0 \\ 
\end{array}
\right] \\
&=
\left[
\begin{array}{ccccc}
(a^{2} + b^2)s_{1} & ab s_{2} & \hdots & ab s_{n} & \delta \\
abs_{2}  & -2ab \a_{2} & \hdots & 0 & 0\\
\vdots & \vdots & \ddots & \vdots & \vdots\\
ab s_{n} & 0 & \hdots & -2ab \a_{n} & 0 \\
\frac{1}{\displaystyle \delta} & 0 & 0 \hdots & 0 & 0
\end{array}
\right] 
\end{aligned}
\end{equation}
Define 
\[G:=\mbox{Diag}(I, \frac{1}{\delta}), \]
where $I$ is the $n \times n$ identity matrix.
Now,
\begin{equation} \label{GQG}
\begin{aligned}
G^{-1}QG &= 
\left[
\begin{array}{ccccc}
I & 0 \\
0 & \delta \\ 
\end{array}
\right] 
\left[
\begin{array}{ccccc}
W^{-1}U'DUW & \delta e_{1} \\
\frac{1}{\delta} e_{1}' & 0 \\ 
\end{array}
\right] \left[
\begin{array}{ccccc}
I & 0 \\
0 &  \frac{1}{\delta} \\ 
\end{array}
\right] \\
&= 
\left[
\begin{array}{ccccc}
W^{-1}U'DUW & e_{1} \\
 e_{1}' & 0 \\ 
\end{array}
\right].
\end{aligned}
\end{equation}
As $W^{-1}U'DUW$ is symmetric (see $(\ref{m4})$), $G^{-1}QG$ is symmetric. Thus, $\widetilde{D}$
is similar to a symmetric matrix. 
Therefore, $\widetilde{D}$ has only real eigenvalues.

We now claim that $\widetilde{D}$ has exactly one positive eigenvalue.
Let $x=(x_{1},\dotsc,x_{n})'$ be orthogonal to $e_1$. Because $x_{1}=0$, from $(\ref{m4})$, we have 
\begin{equation} \label{eperp}
x'W^{-1}U'DUWx=-2ab\sum_{i=2}^{n} x_{i}^{2} \a_{i} \leq 0. 
\end{equation}
The subspace
\[\nabla:=\{(x,r): x' e_{1}=0, ~r \in \rr \} \]
of $\rr^{n+1}$ has dimension $n$.
Consider a vector  $w:=(x,r)' \in \nabla$.  By $(\ref{GQG})$ and $(\ref{eperp})$,
\begin{equation*}
\begin{aligned}
w'G^{-1}QGw &= x'W^{-1}U'DUWx + r e_1'x \\
&=x'W^{-1}U'DUWx \leq 0.\\
\end{aligned}
\end{equation*}
Hence, $G^{-1}QG$ has at least $n$ non-positive eigenvalues. Because $D$ has exactly one positive eigenvalue, $G^{-1}QG$ has
at least one positive eigenvalue. So, $G^{-1}QG$ has exactly one positive eigenvalue.  Because $\widetilde{D}$ and $G^{-1}QG$ are similar, $\widetilde{D}$ has exactly one positive eigenvalue. The result is proved.
\end{proof}

\subsection{Sign pattern of $(\rho(D)I-D)^r$}
Let the eigenvalues of $D$ be $\delta_1,\dotsc,\delta_n$.
Since $D$ is a nonnegative matrix and has exactly one positive eigenvalue, $\rho(D)$ is the only positive eigenvalue of $D$.  Therefore,
\[\gamma_i:=\rho(D)-\delta_i \geq 0~~~i=1,\dotsc,n.\]  
Now,  let $A$ be an invertible matrix such that \[A^{-1}D A=  \Diag(\delta_1,\dotsc,\delta_n).\]
Define 
\[S:= \rho(D)I-D.\]
Then, for any $r>0$
$$S^r:=A\Diag(\gamma_1^r,\dotsc,\gamma_n^r) A^{-1}. $$ 
As a consequence of Theorem \ref{ww}, we now have the following result.
\begin{corollary}
If $0<r<1$, then
$S^r$ is an ${\bf M}$-matrix. 
\end{corollary}
\begin{proof}
Fix  $0 < r < 1$.
Since $S$ is an ${\bf M}$-matrix, by Theorem $3.1$ in \cite{Ando}, $
S^r$ is an ${\bf M}$-matrix.
The proof is complete.
\end{proof}

To illustrate,  we give the following example.

\begin{example}
Consider the matrix $D$ given in Example \ref{gedmexample}.  The eigenvalues of $D$ are approximately
\[-24,  -36.1322~\mbox{and}~100.1322.\]
Now the matrix $S$ is
\[S = \rho(D)I-D =  \left[\begin{array}{ccc}
88.1322 & -36 & -51 \\
-36 & 88.1322 & -51 \\
-43 & -43 & 84.1322\end{array}\right].\]
Then
\[
S^{\frac{1}{2}} =  \left[\begin{array}{ccc}
7.8037 & -3.3378 & -4.3690 \\
-3.3378 & 7.8037 & -4.3690 \\
-3.6836 & -3.6836 & 7.2073
\end{array}\right].
\]
It is easy to see that $S^{\frac{1}{2}}$ is an ${\bf M}$-matrix.
\end{example}

\subsection{Null space of a $\GDM$}
The main result about the null space of a $\GDM$ is $\nul(D) \subseteq \1^\perp$; so $\1 \in \range(D)$ and this in turn will be useful
to investigate the Moore-Penorse inverse of $D$.
\begin{theorem} \label{nuld}
If $D$ is a $\GDM$, then $\1 \in \nul(D)$.
\end{theorem}
\begin{proof}
Following same notation as in Theorem \ref{ww}, we have
by equation $(\ref{pdp})$, 
$$\Delta:=U'DU=
\left[
\begin{array}{ccccc}
(a^{2} + b^2)s_{1} & b^2 s_{2} & \hdots & b^2 s_{n} \\
a^{2} s_{2} & -2ab \a_{2} & \hdots & 0\\
\vdots & \vdots & \ddots & \vdots \\
a^2 s_{n} & 0 & \hdots & -2ab \a_{n}
\end{array}
\right].$$ 
As
$\Delta_{11}=(a^2+b^2) s_{1}$ and $s_1=\tr(L)$, $\Delta_{11}>0$. 
Let $Dx=0$. We claim that $\1' x=0$. 
Since $U \Delta U'=D$ and $Dx=0$, we have 
\begin{equation} \label{Delta}
\Delta U'x=0 .
\end{equation}
If
$v:=(s_{2},s_{3},\dotsc,s_{n})'$ and
$S:=\Diag(2ab\lambda_2,\dotsc,2ab\lambda_n)$, then
\[ \Delta=\left[
\begin{array}{cccc}
\Delta_{11} & b^2v' \\
a^2 v & -S
\end{array}
\right] .\] Put $y:=U'x$. By writing $y=(y_{1},\bar{y})'$, where $y_1 \in \rr$ and $\bar{y}' \in \rr^{n-1}$,
from $(\ref{Delta})$, we have 
\begin{equation} \label{e11}
y_{1} \Delta_{11} + b^2 v' \bar{y}=0
\end{equation}
\begin{equation} \label{e22}
y_1 a^2 v=S \bar{y} .
\end{equation}
If possible, let $y_1 \neq 0$.
Then from $(\ref{e22})$ we have, \[v=\frac{1}{a^2y_{1}} S \bar{y}. \]
Thus by equation $(\ref{e11})$, 
\[y_1 \Delta_{11} + \frac{b^2}{a^2 y_{1}} \bar{y}'S \bar{y}=0,\]
or equivalently,
\begin{equation} \label{ss1}
y_{1}^{2} \Delta_{11} + \frac{b^2}{a^2} \bar{y}' S \bar{y}=0.  
\end{equation}
As $S$ is positive semidefinite, $\bar{y}'S \bar{y}$ is nonnegative. Because
$\Delta_{11}>0$,
\[y_{1}^{2} \Delta_{11} + \frac{b^2}{a^2} \bar{y}' S \bar{y} >0, \]
contradicting
$(\ref{ss1})$.
 Hence, $y_1=0$. 
Since $Uy=x$ and $y=(0,y_2,\dotsc,y_n)'$, 
\[x \in \span\{Ue_2,\dotsc,Ue_n\}=\1^\perp.\] So,
$\1'x=0$. The proof is complete.
\end{proof}

\begin{corollary}
$\1 \in \nul(D')$.
\end{corollary}
\begin{proof}
As $D'$ is also a $\GDM$,
$\1 \in \nul(D')$.
\end{proof}

\begin{corollary} \label{range}
$\1 \in \range(D) \cap \range(D')$.
\end{corollary}
\begin{proof}
This follows easily since $\nul(D)=\range(D')^{\perp}$.
\end{proof}

\begin{corollary} \label{1'D1=0}
$\1' D^{\dag} \1=0$ if and only if there exists $f \in \1^{\perp}$ such that $Df=\1$. 
\end{corollary}
\begin{proof}
Suppose $\1' D^{\dag} \1=0$. Let $Dx=\1$. Since
$DD^{\dag}\1=\1$, we have $x-D^{\dag} \1 \in \nul(D)$ and because $\nul(D) \subseteq \1^{\perp}$,   
$\1' (x-D^{\dag}) \1=0$. So, $(\1'x) \1=0$; hence $\1' x=0$. 
Conversely, let $Df=\1$ and $\1'f=0$.  Since, $DD^{\dag}\1=\1$ and 
$\nul(D) \subseteq \1^{\perp}$, $f-D^\dag \1 \in \1^{\perp}$. 
Thus, $\1' D^\dag \1=0$. 
\end{proof}

\begin{corollary} \label{cor5}
The following are equivalent.
\begin{enumerate}
\item[\rm (a)] $\1' D^{\dag} \1 \neq 0$.
\item[\rm (b)] $\nul(D)=\nul(D')=\nul(L) \cap \1^{\perp}$.
\end{enumerate}
\end{corollary}

\begin{proof}
Assume (a). To prove (b),
it suffices to show that 
\[\nul(D)=\nul(L) \cap \1^{\perp}. \]
Let $x \in \nul(D)$. By Theorem $\ref{nuld}$, $x \in \1^{\perp}$. Since
\begin{equation*}
\begin{aligned}
x'Dx &=a^2 x' \L J x +b^2x' J \L x -(2ab) x'Lx \\
&= -2ab x' Lx=0.
\end{aligned}
\end{equation*}
 Because $x'Lx=0$ and $L$ is positive semidefinite, $Lx=0$. So, $x \in \nul(L)$. Thus, 
\[\nul(D) \subseteq \nul(L) \cap \1^\perp.\] Now, let $f \in \nul(L) \cap \1^{\perp}$. Then,
\begin{equation} \label{eqdf}
\begin{aligned}
Df&=a^2 \L J f+ b^2 J \L f-2ab Lf \\
&=b^2J \L f =b^2 \11' \L f 
=b^2(\1' \L f) \1.
\end{aligned}
\end{equation}
From Corollary \ref{1'D1=0}, we find that $\1'\L f =0$.  
Thus, $f \in \nul(D)$. This proves (a) $\Rightarrow$ (b).

Assume (b). If $\1'D^{\dag} \1=0$, then by Corollary \ref{1'D1=0}, there exists $f \in \1^\perp$ such that 
$Df=\1$. Therefore, $f'Df=0$ and this gives $f'Lf=0$. Since $L$ is positive semidefinite, $Lf=0$ and therefore by our assumption, $f \in \nul(D)$
contradicting $Df=\1$.
 Thus (b) $\Rightarrow$ (a). The proof is complete. 
\end{proof}

\begin{corollary} \label{nullequal}
$\nul(D)=\nul(\L J + J \L-2 L)$.
\end{corollary}
\begin{proof}
Put $$E:=\L J + J \L -2L.$$
Let $x \in \nul(D)$. Then by Theorem $\ref{nuld}$,
$x \in \1^\perp$. So,
\begin{equation} \label{dx}
\begin{aligned}
0=Dx &=a^2 \L J x + b^2 J \L x -2 a b L x \\
&=b^2 J \L x -2a b Lx \\
&= b^2 \1 \1' \L x  -2 a b L x.
\end{aligned}
\end{equation}
Because $Lx$  is an element in $\1^\perp$, we get by $(\ref{dx})$,
\begin{equation} \label{1'Lx}
\1' \L x=0 ~~\mbox{and}~~ Lx=0.
\end{equation}
On the other hand, we have
\begin{equation*}
\begin{aligned}
Ex =\L J x + J \L x -2 L x 
&=J \L x -2Lx \\
&=  \1 \1' \L x  -2 L x.
\end{aligned}
\end{equation*}
By $(\ref{1'Lx})$,  it now follows that
$Ex=0$. Thus, $x \in \nul(E)$. So, $$\nul(D) \subseteq \nul(E).$$ Similar argument leads to
$\nul(E) \subseteq \nul(D)$. The proof is complete.
\end{proof}

\subsection{Moore-Penrose inverse of a $\GDM$}
We now obtain some properties of the Moore-Penrose inverse of $D$. The following result says that $\1' D^{-}\1$ is invariant for any choice of $g$-inverse of $D$. 

\begin{theorem}
If $D$ is a $\GDM$, and if $D^-$ is a generalized inverse of $D$, then $\1' D^- \1=\1 D^\dag \1$.
\end{theorem}
\begin{proof}
As $\1$ is an element in the column space of $D$ and $D'$, we have
\[\1' DD^\dag=\1'~~\mbox{and}~~DD^{\dag} \1=\1. \]
So, \[\1' D^{-} \1=\1' D^\dag D D^{-} DD^{\dag} \1.\]
As $D D^{-}D=D,$ we get 
\[\1' D^- \1=\1' D^\dag D D^\dag \1=\1' D^\dag \1. \]
\end{proof}

\begin{theorem} \label{1D1}
$1' D^{\dag} \1 \geq 0$.
\end{theorem}
\begin{proof}
In view of equation (\ref{GQG}), $\widetilde{D}$ is similar to 
\begin{equation}
S:=\left[
\begin{array}{cccccc}
W^{-1}U'DUW & e_1 \\
e_1' & 0
\end{array}
\right],
\end{equation}
which is symmetric. 
We claim the following. \\
{\bf Claim:} There exists $x \in \rr^n$ such that $W^{-1}U'DUWx=e_1$. 

Since $UWe_1 \in \span\{\1\}$ and $\1$ is in the column space of $D$,  there exists $y \in \rr^n$ such that
$Dy=UWe_1$. As $U$ and $W$ are non-singular, $y=UWx$ for some $x \in \rr^n$. Thus,
$DUWx=UWe_1$ and the claim is true.

Applying inertia formula to $S$ (see (N\ref{(N61)})), we have 
\[\inertia(S)=\inertia(W^{-1}U'DUW)+ \inertia(-e_1'W^{-1} U' D^{\dag} UW e_1). \]
By corollary \ref{bordered}, $\widetilde{D}$ has exactly one positive eigenvalue.
Since $W^{-1}U'DUW$ has exactly one positive eigenvalue, by the above formula,
\[-e_1'W^{-1}U'D^{\dag}UWe_1 \leq 0. \]
From the definition of $U$ and $W$, we see that
\[-e_1'W^{-1}U'D^\dag UWe_1=-\1'D^{\dag}\1, \]
and hence $\1'D^\dag \1 \geq 0$. This proves the result.
\end{proof}

\begin{theorem} \label{pd+p}
$-PD^\dag P$ is positive semidefinite.
\end{theorem}
\begin{proof}
We first claim that $PD^\dag P$ is symmetric. Let $U$ be the orthogonal matrix given in Theorem $\ref{ww}$. Define
\begin{equation} \label{f}
f_i:=Ue_i~~~i=2,\dotsc,n. 
\end{equation}
Each $f_i \in \1^\perp$.  To complete the proof of the claim, it suffices to show that
\[f_{i}'D^\dag f_j=f_{j}'D^\dag f_i~~~i,j=2,\dotsc,n.\]
 Fix $i \neq j$ and $i,j \geq 2$. 
We know from Theorem $\ref{ww}$ that $W^{-1}U'D^\dag UW$ is symmetric. Hence,
\begin{equation} \label{equal}
e_{i}'W^{-1}U'D^\dag U W e_{j}=e_{j}'W^{-1}U'D^\dag U We_{i}. 
\end{equation}
Since \[e_{i}'W^{-1}= e_{i}' ~~\mbox{and}~~We_{j}=e_{j}~~~\mbox{for all}~i,j \geq 2,\] we have
\[e_i'W^{-1}U'D^\dag UW e_j=e_i'U'D^\dag Ue_j; \]
hence by $(\ref{f})$,
\[ e_i'W^{-1}U'D^\dag UW e_j=f_i'D^\dag f_j.\]
By a similar reasoning, 
\[ e_{j}'W^{-1}U'D^\dag U We_{i}=f_{j}'D^\dag f_i.\]
In view of $(\ref{equal})$,
\[f_i'D^\dag f_j=f_{j}'D^\dag f_i. \]
Thus, $PD^\dag P$ is symmetric.

Consider the bordered matrix
\[\S:=\left[
\begin{array}{cccccc}
W^{-1}U'D^{\dag}UW & e_1 \\
e_{1}' & 0
\end{array}
\right].
\]
Because $\1 \in \range(D)$,  $e_1 \in \range(W^{-1} U'D^{\dag}UW)$. Since $W^{-1}U'D UW$  is symmetric,
 $(W^{-1}U'DUW)^\dag$ is symmetric as well. Because
$$(W^{-1}U'DU W)^\dag=W^{-1}U'D^\dag UW,$$ $\S$ is symmetric.
By inertia formula in (N\ref{(N61)}), we have 
\begin{equation} \label{inertiafo}
\inertia(-e_1'W^{-1}U'DU W e_1)+\inertia(W^{-1}U'D^\dag U W)=\inertia(\S). 
\end{equation}
By an easy computation, we see that $$e_1'W^{-1}U'DUWe_1=\frac{1}{n}\1'D \1.$$ So,
\[\inertia(-e_1'W^{-1}U'DUWe_1)=(1,0,0). \]
Because $W^{-1}U'D^\dag U W$ has exactly one positive eigenvalue, it follows from $(\ref{inertiafo})$ that
$\S$ has exactly one positive eigenvalue. By (N\ref{(N8)}),  
\[x \in e_1^\perp \implies x'W^{-1}U'D^{\dag}UW x \leq 0.\]
Specializing $$x=e_i~~~i=2,\dotsc,n$$ in the above inequality leads to $$f_{i}'D^{\dag}f_i \leq 0~~i=2,\dotsc,n.$$ This proves 
$-PD^\dag P$ is positive semidefinite. The proof is complete.
\end{proof}

\begin{corollary}
$D^\dag$ is negative semidefinite on $\1^\perp$.
\end{corollary}
\begin{proof}
Let $x \in \1^\perp$. Then, $x'D^\dag x=x'P D^\dag P x$; hence by Theorem \ref{pd+p}, $x'D^\dag x \leq 0$.
\end{proof}

\begin{theorem} \label{eab}
Let $E=\L J+ J \L -2 L$. Then,
\[\1' D^{\dag} \1>0 ~~ \mbox{if and only if}~~\1' E^{\dag} \1>0. \]
\end{theorem}
\begin{proof}
Suppose $\1' D^\dag \1>0$. If possible, let $\1' E^\dag \1=0$. We shall now get a contradiction.
By Corollary \ref{1'D1=0},  there exists $f \in \1^\perp$ such that
\begin{equation} \label{ealpha}
E f=\1.
\end{equation}
As
$f' Ef=0$,  we have $f'Lf=0$; so $L f=0$. 
We now have
\begin{equation} \label{DDE}
\begin{aligned}
Df &= a^2\L J f + b^2 J \L f-2ab Lf  \\
&= b^2\1 \1' \L f=b^2(\1' \L f) \1.
\end{aligned}
\end{equation}
Since $\1' D^\dag \1>0$, from (\ref{DDE}) and Corollary \ref{1'D1=0}, we find that  $\1' \L f = 0$. Thus, $f \in \nul(D)$. 
By Corollary \ref{nullequal}, $f \in \nul(E)$. This contradicts $(\ref{ealpha})$. So, $\1' E^\dag \1>0$. 

By using a similar argument, we get the reverse implication. 
\end{proof}

\subsection{Generalized Circum Euclidean distance matrix}
Suppose $p^1,\dotsc,p^n$ are some vectors in a Euclidean space $(V,\langle \cdot, \cdot \rangle)$. If there exists a vector $v \in V$ and $r>0$ such that
$$\langle p^i-v, p^i-v\rangle =\| p^i-v\|^2=r~~\mbox{for all}~i=1,\dotsc,n,$$ then we say that $p^1,\dotsc,p^n$ lie on the surface of a hypersphere. Now the Euclidean distance matrix $[\|p^i-p^j\|^2]$ is called a circum $\EDM$.
A well-known result (\cite[Theorem $3.4$]{{Tarazaga}}) characterizes all circum $\EDM$s. 
This says that $E$ is a circum $\EDM$ if and if there exists a vector $s$ and a scalar $\beta$ such that
\[Es=\beta \1~~\mbox{and}~~s'\1=1. \]
This is equivalent to saying that $E$ is a circum $\EDM$ if and only if $\1'E^\dag \1>0$. 
We now introduce the following definition.
\begin{definition} \rm
We say that $D=[d_{ij}]$ is a circum $\GDM$ if there exist $a,b>0$ and vectors $x^1,\dotsc,x^n$ on the surface of a hypersphere 
such that
\[d_{ij}=\|ax^i-bx^j\|^2, \] where $\sum_{i=1}^{n} x^i=0$.
\end{definition}

We now have the following result for $\GDM$s. 

\begin{theorem}
The following are equivalent.
\begin{enumerate}
\item[\rm (1)] 
$D$ is a circum $\GDM$.

\item[\rm (2)]  $\1' D^{\dag} \1>0$. 
\end{enumerate}
\end{theorem}
\begin{proof}
As $D=[d_{ij}]$ is a $\GDM$, 
there exist vectors $x^1,\dotsc,x^n$ such that
\[\sum_{j=1}^{n} x^i=0 ~~\mbox{and}~~d_{ij}=\|ax^i-bx^j\|^2.\] 
Define
\[ e_{ij}:=[\|x^i-x^j\|^{2}]~~\mbox{and}~~E:=[e_{ij}].\]
Assume (1).  
Then, $x^1,\dotsc,x^n$ lie on the surface of a hypersphere. So, $E$ is a circum $\EDM$ and hence $\1' E^\dag \1>0$.
By Theorem $\ref{eab}$, $\1' D^\dag \1>0$. 
This proves (2).

Assume (2). 
Then by Theorem $\ref{eab}$, $\1'E^\dag \1>0$. 
So, $x^1,\dotsc,x^n$ lie on the surface of a hypersphere. Hence $D$ is a circum $\GDM$.
\end{proof}

\begin{corollary} \label{nulld}
If $D$ is a circum $\GDM$, then $\rank(D)=\rank(L)+1$; otherwise $\rank(D)=\rank(L)+2$.
\end{corollary}
\begin{proof}
Let $E=\L J + J \L -2 L$. By Corollary \ref{eab},
$\rank(D)=\rank(E)$. By the previous result $D$ is a circum $\GDM$ if and only if $\1' D^\dag \1>0$. By Theorem \ref{eab},
$\1' E^\dag \1>0$  if and only if $\1' D^\dag \1 >0$. In view of
Proposition $1$ in \cite{Kurata}, we get
\[\rank(E) = \begin{cases} \rank(L)+1 & \mbox{if}~ \1' E^\dag \1 >0 \\ 
\rank(L)+2 & \mbox{else}. \end{cases}\]
The proof now follows easily.
\end{proof}

We now obtain a formula to compute the Moore-Penrose inverse of a circum $\GDM$.

\begin{theorem}\label{mpinverse}
If $D$ is a circum $\GDM$, then
\[D^{\dag}=-\frac{1}{2ab} L^{\dag} + \frac{1}{\1' D^{\dag} \1} (D^{\dag} \1)(\1'D^{\dag}).\]
\end{theorem}
\begin{proof}
Let $$\a:=\1'D^{\dag} \1,~S:=D^{\dag}- \frac{\displaystyle 1}{\displaystyle \a}(D^{\dag} \1)(\1' D^{\dag})~~\mbox{and}~K:=PDP.$$
We now prove that
\(SKS=S \). Let $x \in \rr^n$. Then, $$x=c_{1} \1+ c_{2} f,$$ for some $c_{1},c_2 \in \rr$ and $f \in \{\1\}^{\perp}$.
Since $S\1=0$, we see that $$Sx=c_{2} Sf.$$ As $P \1=0$ and $Pf=f$, $$SPx=c_{2} Sf.$$ Therefore, $SP=S$. In a similar manner, by using
$\1' S=0$ and $\1'P=0$, we get $PS=S$.  Thus, to prove $SKS=S$, it suffices to show that $SDS=S$.  As $\1'S=0$,  we note that
\begin{equation}
\begin{aligned}
SDS &= (D^{\dag} - \frac{1}{\a} (D^{\dag} \1)(\1' D^{\dag}) )D S \\
&=D^{\dag}DS\\
&=D^{\dag}D(D^{\dag} - \frac{1}{\a} (D^{\dag} \1)(\1' D^{\dag}) )\\
&=S.
\end{aligned}
\end{equation}
Hence, $SKS=S$. We claim that 
\(KSK=K.\) Again by using 
$PS=SP=S$ and $P \1=0$, we see that
\begin{equation}
\begin{aligned}
KSK &=(PDP)S(PDP) \\
&= (PDS)(PDP) \\
&=PDSDP \\
&=PD(D^{\dag}-\frac{1}{\a} D^{\dag}\1 \1' D^{\dag})DP \\
&=PDP \\
&=K.
\end{aligned}
\end{equation}
Since $SK=KS=P$, we conclude that $K$ is the Moore-Penrose inverse of $S$. Thus, 
$$(PDP)^{\dag}=S. $$ Since $L \1=0$ and $L$ is symmetric,  $PL=LP=L$ and hence  
\begin{equation}
\begin{aligned}
PDP&=-2ab PLP \\
&=-2ab L.
\end{aligned}
\end{equation}
So, \[(PDP)^{\dag}=-\frac{1}{2ab} L^{\dag}. \]
Thus,
\[S=(PDP)^{\dag}=D^{\dag}-\frac{1}{\a}(D^{\dag} \1)(\1' D^{\dag}).\] This completes the proof
of the formula
\[D^{\dag}=-\frac{1}{2ab}L^{\dag} + \frac{1}{\1' D^{\dag} \1}(D^{\dag} \1)(\1' D^{\dag}) .\]
\end{proof}

\subsection{Some majorization results}
Suppose all the eigenvalues of an $n \times n$ matrix $A$ are real. Let the eigenvalues of $A$ be $\lambda_1(A),\dotsc,\lambda_{n}(A)$,
where
\[\lambda_{1}(A) \geq \cdots \geq \lambda_{n}(A). \]
We now use $\lambda(A)$ to denote the vector $(\lambda_1(A),\dotsc,\lambda_n(A))'$.
In the following, we obtain a Schur-type majorization result for $\GDM$s.

\begin{theorem}
 $\diag(D)$ is majorized by $\lambda(D)$. 
\end{theorem}
\begin{proof}
In view of (\ref{pdp}), 
\begin{equation}
U'DU=\left[\begin{array}{cccc}
(a^2+b^2) \tr(L) & b^2 s \\
a^2 s & -2ab \mbox{Diag}(\a_2,\dotsc,\a_n)
\end{array}
\right].
\end{equation}
Define \[x:=((a^2+b^2) \tr(L),-2ab \a_2,\dotsc,-2ab \a_n)'.\] Without loss of generality, we can assume that
\[ -2ab \a_2 \geq \cdots \geq -2ab \a_n~~\mbox{and}~~-l_{11} \geq \cdots \geq -l_{nn}. \]
We now prove the following. \\
{\bf Claim:} 
 $(a-b)^2 \diag(L) \prec x.$ \\
Set $\a_1=0$. 
By the majorization result of Schur,
\[-2ab \diag(L) \prec -2ab \lambda(L). \]
So, for each $1 \leq k \leq n-1$,
\begin{equation} \label{lii}
-2ab \sum_{i=1}^{k} l_{ii} \leq -2ab \sum_{i=1}^{k} \alpha_k.
\end{equation}
As $L$ is positive semidefinite,  for each $1 \leq k \leq n-1$,
\begin{equation} \label{ljj}
(a^2+b^2) \sum_{i=1}^{k} l_{ii} \leq (a^2+b^2) \tr(L).
\end{equation}
Using $(\ref{lii})$ and (\ref{ljj}), for each $1 \leq k \leq n-1$, we find that
\[(a-b)^2 \sum_{i=1}^{k} l_{ii} \leq (a^2+b^2) \tr(L)-2ab \sum_{i=1}^{k} \a_i. \]
Furthermore,
\[(a-b)^2 \sum_{i=1}^{n} l_{ii} = (a^2+b^2) \tr(L)-2ab \sum_{i=1}^{n} \a_i.\]
Hence, $$(a-b)^2 \diag(L) \prec x.$$ This proves the claim. 

By an easy verification, $$\diag(D)=(a-b)^2 \diag(L).$$ Therefore, 
\begin{equation} \label{maj1}
\diag(D) \prec x. 
\end{equation}
We now recall equation (\ref{m4}):
\begin{equation} 
W^{-1}U'DUW=
\begin{array}{ccccc}
\left[
\begin{array}{ccccc}
(a^{2} + b^2)s_{1} & ab s_{2} & \hdots & ab s_{n} \\
abs_{2}  & -2ab \a_{2} & \hdots & 0\\
\vdots & \vdots & \ddots & \vdots \\
ab s_{n} & 0 & \hdots & -2ab \a_{n}
\end{array}
\right].
\end{array}
\end{equation}
Since $W^{-1}U'DUW$ is symmetric and $\diag(W^{-1}U'DUW)=x$, by Schur majorization result,
\begin{equation} \label{maj2}
x \prec \lambda(W^{-1}U'DUW)=\lambda(D). 
\end{equation}
By (\ref{maj1}) and $(\ref{maj2})$, 
 $$\diag(D) \prec \lambda(D).$$ The proof is complete.
\end{proof}

We now prove another result.
\begin{theorem}
$\lambda(D) \prec \lambda(\frac{D+D'}{2})$.
\end{theorem}
\begin{proof}
By (\ref{pdp}),
\[
U'DU=\left[\begin{array}{cccc}
(a^2+b^2) \tr(L) & b^2 s \\
a^2 s & -2ab \mbox{Diag}(\a_2,\dotsc,\a_n)
\end{array}
\right].
\] 
Hence,  \[
U'(D+D')U= \left[
\begin{array}{ccccc}
2(a^2+b^2) \tr(L) & (a^2+b^2)s \\
 (a^2+b^2)s & -4ab \mbox{Diag}(\a_2,\dotsc,\a_n) \\ 
\end{array}
\right].
\]
By $(\ref{m4})$, 
\[
W^{-1}U'DUW=
\begin{array}{ccccc}
\left[
\begin{array}{ccccc}
(a^{2} + b^2)\tr(L) & ab s \\
abs  & -2ab \Diag(\a_{2},\dotsc,\a_n) \\
\end{array}
\right].
\end{array}
\] 
Put $$F:=W^{-1}U'DUW.$$ 
Since $$ab \leq \frac{a^2+b^2}{2},$$
 we can write
\[U'(D+D')U=2F+ \left[
\begin{array}{cccc}
0 & \alpha s \\
\alpha s & 0 \\
\end{array}
\right], \]
for some $\alpha>0$.
Put $$Q:=\left[
\begin{array}{cccc}
0 &  s \\
 s & 0 \\
\end{array}
\right].$$ We now have 
\[U'(D+D')U-\alpha Q=2F. \]
By Ky Fan majorization theorem,
\[\lambda(U'(D+D')U-\alpha Q) \prec \lambda(D+D')-\alpha \lambda(Q).\]
As,
\[\lambda(Q)=(\lambda_1(Q),0,\dotsc,0,-\lambda_1(Q)) ~~\mbox{and}~~\lambda_1(Q)>0,\]  we have 
\[\lambda(D+D')-\a \lambda(Q) \prec \lambda(D+D'). \]
Thus,
\[\lambda(U'(D+D')U-\alpha Q) \prec \lambda(D+D'), \]
i.e. 
\[\lambda(2F) \prec \lambda(D+D'). \]
Since  $$\lambda(2F)=\lambda(2D),$$ we get $$\lambda(D) \prec \lambda(\frac{D+D'}{2}).$$
This completes the proof.
\end{proof}

As an immediate corollary of the above result, we have the following.

\begin{corollary}
$\rho(D) \leq \rho(\frac{D+D'}{2})$.
\end{corollary}

\section{Application}
We end this paper with an application.
\subsection{Constructing infinitely divisible matrices} 
Generalized distance matrices can be used to construct infinitely divisible matrices. Recall that a symmetric matrix $E=[e_{ij}]$ is infinitely divisible if $E$ is an entry-wise nonnegative and $[e_{ij}^{r}]$ is a positive semidefinite matrix for all $r\geq 0$. 
\begin{theorem}
Let $S=[s_{ij}]$ be an $n \times n$ generalized Laplacian matrix. If $a,b>0$, define
\[d_{ij}:=a^2 s_{ii} + b^2 s_{jj} -2ab s_{ij},~~f_{ij}:=\max{(d_{ij},d_{ji})}; \]
\[D:=[d_{ij}] ~~\mbox{and}~~ F:=[f_{ij}].\] If $\rank(S)=n-1$, then each $f_{ij}>0$ and 
 $[\frac{1}{f_{ij}}]$ is an infinitely divisible matrix. 
\end{theorem}
\begin{proof}
We claim that $d_{ij} >0$ for all $i,j$. Since $S \1=0$, all cofactors of $S$ are equal. Hence all principal minors of $S$ are positive. In particular,
every $2 \times 2$ principal submatrix of $S$ is positive definite. So, $d_{ij}>0$ and thus $f_{ij}>0$ for all $i,j$.

Put $G:=[f_{ij}]:=[\max(d_{ij},d_{ji})]$.
 By Theorem 4.2.9 in Bapat \cite{mp}, it suffices to show that $G$ has exactly one
(simple) positive eigenvalue.  We shall use the following identity: If $\a$ and $\b$ are positive, then
\[\max{(\alpha, \b)}=\frac{1}{2}[(\a + \b)+|\alpha -\beta |].\]
Hence,  
$$2G=[d_{ij}+d_{ji}] + [|d_{ij}-d_{ji}|]$$ has exactly one simple positive eigenvalue.
Put $$A:=[d_{ij}+d_{ji}]~~\mbox{and}~~B:=[|d_{ij}-d_{ji}|].$$ If $D=[d_{ij}]$, then for any $x \in \1^{\perp}$,  
\[x'Dx=-2ab (x'Sx) \leq 0, \]
and hence $x'Ax \leq 0$. We now claim that $B$ is negative semidefinite on $\1^{\perp}$ as well.
It can be noted easily that 
\[|d_{ij}-d_{ji}|=|a^2-b^2| |s_{ii}-s_{jj}|.  \]
If $\a,\b \geq 0$, then 
\[|\a-\b |=\a + \b -2 \min(\a,\b), \]
and therefore,
$$|d_{ij} -d_{ji}|=|a^2-b^2| (s_{ii} + s_{jj}-2\min(s_{ii},s_{jj})). $$
It is well known that $\min(s_{ii},s_{jj})$ is positive semidefinite. 
So,
$B$ is negative semidefinite on $\1^{\perp}$, and consequently, $2G=A+B$ is negative semidefinite on $\1^{\perp}$. 
Thus, $G$ has at least $n-1$ non-positive eigenvalues. Since diagonal entries of $G$ are positive, $G$ has at least 
one positive eigenvalue. Thus, $G$ has exactly one simple positive eigenvalue. This completes the proof.
\end{proof}

\section*{Funding}
The second author acknowledges the support of the Indian National Science Academy under the INSA Senior Scientist scheme.


\begin{thebibliography}{99}

\bibitem{Ando}
Ando T.  Inequalities for $M$-matrices.  Linear Multilinear Algebra. 1980;8:291--316.

\bibitem{mp}
Bapat RB.  Multinomial probabilities, permanents and a conjecture of Karlin and Rinott.  Proc Am Math Soc. 1988;102(3):467--472.

\bibitem{Bapat}
Bapat RB, Raghavan TES.  Nonnegative matrices and applications, Encyclopedia of Mathematics and its Applications. Cambridge: Cambridge University Press; 1997.


\bibitem{Bhatia}
Bhatia R,  Jain T.  Mean matrices and conditional negativity. Electron J Linear Algebra. 2016;29:206--222.


\bibitem{Haynsworth}
Carlson D, Haynsworth E, Markham T. A generalization of the Schur complement by means of the Moore-Penrose inverse. SIAM J Appl Math.
1974;26:169--175.


\bibitem{Ferland}
Ferland JA. Matrix-Theoretic criteria for the quasiconvexity of twice continuously differentiable functions. Linear Algebra Appl. 1981;38:51--63.


\bibitem{Fiedler}
Fiedler M. Elliptic matrices with zero diagonal.  Linear Algebra Appl. 2011;197/198:337--347.

\bibitem{Kurata}
Kurata H, Bapat RB. Moore-Penrose inverse of a hollow symmetric matrix and a predistance matrix. Spec matrices. 2016;4:270--282.

\bibitem{mic}
Micchelli CA. Interpolation of scattered data: distance matrices and conditionally positive definite matrices. Constr Approx. 1986;2:11--22.

\bibitem{Tarazaga}
Tarazaga P, Hayden TL,  Wells J. Circum-Euclidean distance matrices and faces. Linear Algebra Appl. 1996;232:77--96.

\bibitem{Zhang}
Zhang F. Matrix theory, Basic results and techniques. New York: Springer; 1991.

\end{thebibliography}
\end{document}